\documentclass{article} 
\usepackage{amssymb,amsthm,amsmath,amsfonts}
\usepackage{graphicx}
\usepackage{color}
\usepackage{cite}
\newtheorem{theorem}{Theorem}[section]

\newtheorem{corollary}[theorem]{Corollary}
\newtheorem{proposition}[theorem]{Proposition}

\theoremstyle{definition}
\newtheorem{definition}[theorem]{Definition}

\theoremstyle{remark}
\newtheorem{remark}{Remark}

\begin{document}
\title{On hyperholomorphic Bergman type spaces in domains of $\mathbb C^2$}
\small{
\author
{Jos\'e Oscar Gonz\'alez-Cervantes$^{(1)\footnote{corresponding author}}$ and Juan Bory-Reyes$^{(2)}$}
\vskip 1truecm
\date{\small $^{(1)}$ Departamento de Matem\'aticas, ESFM-Instituto Polit\'ecnico Nacional. 07338, Ciudad M\'exico, M\'exico\\ Email: jogc200678@gmail.com\\$^{(2)}$ {SEPI, ESIME-Zacatenco-Instituto Polit\'ecnico Nacional. 07338, Ciudad M\'exico, M\'exico}\\Email: juanboryreyes@yahoo.com
}

\maketitle
\begin{abstract}
Quaternionic analysis is regarded as a broadly accepted branch of classical analysis referring to many different types of extensions of the Cauchy-Riemann equations to the quaternion skew field $\mathbb H$.

In this work we deals with a well-known $(\theta, u)-$hyperholomorphic $\mathbb H-$valued functions class related to elements of the kernel of the Helmholtz operator with a parameter $u \in\mathbb H$, just in the same way as the usual quaternionic analysis is related to the set of the harmonic functions.

Given a domain $\Omega\subset\mathbb H\cong \mathbb C^2$, we define and study a Bergman spaces theory for $(\theta, u)-$ hyperholomorphic quaternion-valued functions introduced as elements of the kernel of ${}^{\theta}_{u}\mathcal D [f]= {}^{\theta} \mathcal D [f] + u f$ with
$u\in \mathbb  H$ defined in $C^1(\Omega, \mathbb H)$, where  
\[
{}^\theta\mathcal D:= \frac{\partial}{\partial \bar z_1} + ie^{i\theta}\frac{\partial}{\partial z_2}j =
\frac{\partial}{\partial \bar z_1} + ie^{i\theta}j\frac{\partial}{\partial \bar z_2},\hspace{0.5cm}  \theta\in[0,2\pi).
\] 
Using as a guiding fact that $(\theta, u)-$hyperholomorphic functions includes, as a proper subset, all complex valued holomorphic functions of two complex variables we obtain some assertions for the theory of Bergman spaces and Bergman operators in domains of $\mathbb C^2$, in particular, existence of a reproducing kernel, its projection and their covariant and invariant properties of certain objects. 

\end{abstract}

\vspace{0.3cm}

\small{
}

\section{Introduction}
The history of Bergman spaces goes back to the book \cite{B} in the early fifties by S. Bergman, where the first systematic treatment of the subject for holomorphic functions of one complex variable was given, and since then there have been a lot of papers on the subject in different function classes. Some standard works here are \cite{BHZ, DS, HKZ} and the references therein, which contain a broad summary and historical notes of the subject, that frees us from referring to missing details. 

Nowadays, quaternionic analysis is regarded as a broadly accepted branch of classical analysis offering a successful generalization of complex holomorphic function theory, the most renowned examples are Sudbery's paper\cite{S} and several already published books see e.g., \cite{GS1}. It relies heavily on results on functions defined on domains in $\mathbb R^4$ with values in the skew field of real quaternions $\mathbb H$ associated to a generalized Cauchy-Riemann operator by using a general orthonormal basis in $\mathbb R^4$ (a so-called structural set $\psi$ of $\mathbb H^4$, see \cite{No1}). This theory is centered around the concept of $\psi-$hyperholomorphic functions, see \cite{MS, S1, S2}.

There has been a great deal of work done over the recent years on weighted Bergman spaces for $\psi-$hyperholomorphic functions in $\mathbb R^4$, see \cite{GM, SV3}. For a natural development of Bergman kernel function in Clifford analysis we refer the reader to \cite{BD, SV4}. 

Despite the fact that theory of $\psi-$hyperholomorphic functions has been generalized and studied extensively to that of $(\psi,u)-$hyperholomorphic functions with a quaternionic parameter $u$ (see \cite{KS} and the references given there), to the best of the authors knowledge, a Bergman theory for such functions classes is still open and this is precisely the main goal of the present work.

The outline of the paper reads as follows. After this short introduction, we briefly review in Section 2 some basic notation and definitions related to quaternionic analysis based on properties of $\psi-$hyperholomorphic functions, where $\psi$ is the structural set. In Section 3 we review the basic
definitions and results of $(\theta; u)-$hyperholomorphic theory and it is also introduced the corresponding Bergman spaces. Section 4 discusses the case of $(\theta; u)-$hyperholomorphic Bergman type spaces in domains of $\mathbb C^2$. Section 5 gives some concluding remarks.

\section{Preliminaries}  
Consider the skew field of real quaternions $\mathbb H$ with its basic elements $1, i, j, k$. Thus any element $x$ from $\mathbb H$ is of the form $x=x_0+x_{1}i+x_{2}j+x_{3}k$, $x_{s}\in \mathbb R, s= 0,1,2,3$. The basic elements define arithmetic rules in $\mathbb H$: by definition $i^{2}=j^{2}=k^{2}=-1$, $i\,j=-j\,i=k; j\,k=-j\,k=i$ and $k\,i=-k\,i=j$. The imaginary unit of $\mathbb C$ is identify with $i$. For $x\in \mathbb H$ we define the mapping of quaternionic conjugation: $x\rightarrow {\overline x}:=x_0-x_{1}i-x_{2}j-x_{3}k$. In this way it is easy seen that $x\,{\overline x}={\overline x}\,x=x^{2}_{0}+x^{2}_{1}+x^{2}_{2}+x^{2}_{3}$. Note that $\overline {xy}={\overline y}\,\,{\overline x}$ for $x,y\in \mathbb H$. We have for $a\in \mathbb C$ and $\overline a$ its complex conjugate, that $a\,j=j\,\overline a$.

In order to deal with the classical theory of holomorphic $\mathbb C-$valued functions of two complex variables, we will identify $\mathbb H$ with $\mathbb C^2$ (or $\mathbb R^4$) by mean of the mapping
\begin{equation} \label{mapping}
x_0+x_{1}i+x_{2}j+x_{3}k\rightarrow (x_0+ix_1)+(x_2+ix_3)j \rightarrow (x_0,x_1,x_2,x_3).
\end{equation}
From now on, we will use this fact in essential way.

The set of elements of the form $q=z_{1}+z_{2}j$, $z_{1},z_{2}\in \mathbb C$, endowed both with a component-wise addition and with the associative multiplication is then another way of stating $\mathbb H$. In particular, the two quaternions $q=z_1+z_2j$ and $\xi=\zeta_1+\zeta_2j$ are multiplied according the rule:
$q\,\xi=(z_1\zeta_1-z_2\bar\zeta_2)+(z_1\zeta_2+z_2\bar\zeta_1)j$. The quaternion conjugation gives: $\overline{z_1+z_2j}:= \bar z_1-z_2j$ and $q\,{\overline q}={\overline q}\, q=|z_{1}|^{2}+|z_{2}|^{2}$.

Quaternionic analysis is regarded as a broadly accepted branch of classical analysis offering a successful generalization of complex holomorphic function theory. It relies heavily on results on functions defined on domains in $\mathbb R^4$ with values in $\mathbb H$ associated to a generalized Cauchy-Riemann operator by using a general orthonormal basis in $\mathbb R^4$ (a so-called structural set $\psi=(1,\psi_1,\psi_2,\psi_3)$ of $\mathbb H^4$, see \cite{No1}). 

This theory is centered around the concept of $\psi-$hyperholomorphic functions, i.e., the collection of all null solutions of the generalized Cauchy-Riemann operator 
$$
{}^\psi\mathcal D=\frac{\partial}{\partial   x_0} +  \psi_1\frac{\partial}{\partial x_1} +  \psi_2\frac{\partial}{\partial x_2} + \psi_3
\frac{\partial}{\partial x_3},
$$ 
see \cite{MS, S1, S2} for more details.

Functions $f$ defined in a bounded domain $\Omega\subset \mathbb R^4\cong \mathbb C^2$ with value in $\mathbb H$ are considered. They may be written as: $f=f_0+f_{1}i+f_{2}j+f_{3}k$, where $f_s, s= 0,1,2,3,$ are $\mathbb R$-valued functions in $\Omega$. Properties as continuity, differentiability, integrability and so on, which as ascribed to $f$ have to be posed by all components $f_s$. We will follow standard notation, for example $C^{1}(\Omega, \mathbb H)$ denotes the set of continuously differentiable $\mathbb H$-valued functions defined in $\Omega$. 

Given  $0\leq \theta < 2\pi$, introduce the left Cauchy--Riemann operator
\[{}^\theta\mathcal D=2\left\{
\frac{\partial}{\partial \bar z_1} + i
e^{i\theta}\frac{\partial}{\partial z_2}j\right\} =2\left\{
\frac{\partial}{\partial \bar z_1} + i
e^{i\theta}j\frac{\partial}{\partial \bar z_2}\right\}
\]
defined on $C^{1}(\Omega, \mathbb H)$, where $z= x_0 + x_1i + ie^{i\theta}j x_2+  e^{i\theta}j x_3 = z_1+   ie^{i\theta}j z_2$  with 
 $z_1=x_0+x_1i$ and $z_2=x_2+x_3 i$. It represents the complex form of the quaternionic operator
\[{}^\theta\mathcal D:= \frac{\partial}{\partial   x_0} + i
 \frac{\partial}{\partial x_1} +   i e^{i\theta} j \frac{\partial}{\partial x_2} +  e^{i\theta} j
\frac{\partial}{\partial x_3},\]
which is ${}^\psi\mathcal D$ associated to the structural set $\psi=(1, i , i e^{i\theta} j, e^{i\theta} j)$.

If ${}^\theta\mathcal D $ acts on the right it is then denoted by ${}^\theta\mathcal D_r$.  These operators operators decompose the four-dimensional Laplace operator:
$${}^\theta\mathcal D\circ \overline{{}^\theta\mathcal D}=\overline{ {}^\theta\mathcal D} \circ {}^\theta\mathcal D =
\overline{{}^\theta\mathcal D_r }\circ {}^\theta\mathcal D_r =
{}^\theta\mathcal D_r \circ \overline{{}^\theta\mathcal
D_r} =\bigtriangleup_{\mathbb R^4}.$$

The elements of the sets ${}^\theta\mathfrak M(\Omega, \mathbb H)=ker \ {}^\theta\mathcal D$ and ${}^\theta \mathfrak M_r(\Omega, \mathbb H)= ker \ {}^\theta\mathcal D_r$ are called left (respectively right) $\theta$-hyperholomorphic functions on $\Omega$.

These spaces ${}^\theta\mathfrak M$ and ${}^\theta \mathfrak M_r$ are right-quaternionic Banach modules, although they are also real linear spaces (from both sides). This has the disadvantage that ${}^\theta\mathcal D_r$ acts on them only as a real linear operator, not a quaternionic linear operator, but in our context it will be good enough since the principal operator under consideration is ${}^\theta\mathcal D$.

In addition,  $f=f_1+f_2j\in {}^\theta\mathfrak M(\Omega, \mathbb H)$, where $f_1,f_2: \Omega\to \mathbb C$, if and only if 
$$
\left\{
\begin{array}{l}
\displaystyle{\frac{\partial f_1}{\partial \bar z_1}  =  
ie^{i\theta}\overline{\frac{\partial f_2
}{\partial \bar z_2}}},\\
 \\
\displaystyle{\frac{\partial f_1}{\partial \bar z_2}  =  
-ie^{i\theta}\overline{\frac{\partial f_2}{\partial \bar z_1}}}.
\end{array} \right.
$$
Let us consider the following spaces of holomorphic maps:
\[
Hol(\Omega,\mathbb C)=\{f\in C^{1}(\Omega, \mathbb C): \displaystyle\frac{\partial f}{\partial \bar z_1}=0, \displaystyle\frac{\partial f}{\partial \bar z_2}=0\}
\]
and more generally
\[
Hol(\Omega,\mathbb C^2)=\{f=(f_1, f_2): f_1, f_2 \in Hol(\Omega,\mathbb C)\}.
\]
The relation
\[\displaystyle{Hol(\Omega,\mathbb C^2)=\bigcap_{0\leq \theta < 2\pi} {}^\theta\mathfrak M(\Omega)}\]
can be found in \cite{MS}.

Given $x,y\in \mathbb H$, denote 
$
 \langle x,y \rangle_{\theta   }=\sum_{k=0}^3 x_k y_k 
$, where $x_{\theta   } =x_0 +x_1 i +x_2 i e^{i\theta} j+ x_3 
e^{i\theta} j $ and $ y_{\theta   } =y_0 +y_1 i +y_2 i e^{i\theta} j+ y_3 
e^{i\theta} j $  with $x_k, y_k\in \mathbb R$ for all $k$.

Set $z_{\theta}= x_0 + x_1i + ie^{i\theta}j x_2+  e^{i\theta}j x_3 = z_1+   ie^{i\theta}j z_2$, for $z\in \mathbb H$, where    
 $z_1=x_0+x_1i, z_2=x_2+x_3 i \in \mathbb C$. The mapping $z_{\theta}\to (z_1,z_2)$ establishes the following  operations in $\mathbb C^2$, according to the quaternionic algebraic structure.  
\begin{itemize}
\item  $zj=j\bar z$ for all $z\in\mathbb C$,
 \item $(z_1, z_2   )  \pm (\zeta_1 , \zeta_2   )=  (z_1 \pm   \zeta_1, z_2\pm
\zeta_2) ,
$ \item $(z_1, z_2   )  ( \zeta _1,   \zeta _2   )=
( z_1 \zeta _1-  \bar{ z}_2  \zeta _2,  \bar z_1 \zeta _2+ z_2   \zeta _1
 ) .$
\item $\overline{ (z_1,  z_2 ) } =( \bar {z}_1 ,  -
z_2)$.
\item $| z_1+ z_2j |^2:= |z_1|_{\mathbb C}^2+|z_2|_{\mathbb C}^2=(z_1,  z_2  )( \bar z_1 ,  - z_2  )$. 
\item Topology in $\mathbb C^2$ is determined by the metric $d(\zeta, z):= |\zeta -z |$ for $z,\zeta\in\mathbb C^2$.
\end{itemize}

For abbreviation we write $z$ instead  of $z_{\theta}$ for the elements of the domains of our functions and in case that we were using the representation  $z=z_1+z_2j$ to indicate that the mapping $ z_1 +z_2 j=x_1+ y_1i+ (x_2 + y_2i)j =x_1+ y_1i+ x_2 j + y_2 k  \in\mathbb H $ we shall write it explicitly. 
		
Let $\Omega\subset \mathbb C^2$ be a bounded domain with its boundary $\partial \Omega$ a compact $3-$dimensional sufficiently smooth hypersurface (co-dimension 1 manifold). The following formulas can be found in many sources (see, for example \cite{GM, MS}).

Given $f, g \in C^1(\overline{\Omega},\mathbb H)$ we have the quaternionic Stokes formulas
\begin{equation} \label{Dif-stokes} d(g\sigma^{\theta}_{\zeta}  f) = 
  g \ {}^{{     \theta   }}\mathcal  D[f]+ \  \mathcal D^{{     \theta   }}[g] f ,
	\end{equation}
	\begin{equation}	 \label{Int-stokes}   \int_{\partial \Omega} g
	\sigma^{\theta}_{\zeta}
	f     =       \int_{\Omega } \left( g    {}^\theta    \mathcal  D[f] +   \mathcal  D^{{     \theta   }}[g] f   \right)  d\mu, 
\end{equation}
and the quaternionic Borel-Pompieu formula  
\begin{align}  &  \int_{\partial \Omega }  ( K_{\theta   }(\zeta - z )\sigma_{\zeta 
}^{\theta} f(\zeta )  +  g(\zeta )   \sigma_{\zeta }^{\theta   } K_{\theta   }(\zeta - z ) ) \nonumber  \\ 
&  - 
\int_{\Omega} (K_{\theta   } ( \zeta - z )
  {}^{\theta   }\mathcal D [f] ( \zeta )  +    \mathcal  D^{{     \theta   }} [g] ( \zeta )   K_{\theta   } ( \zeta - z )
     )d\mu_{\zeta}   \nonumber \\
		=  &    \label{ecua4}   \left\{ \begin{array}{ll}  f( z ) + g( z ) , &   z \in \Omega,  \\ 0 , &   z \in \mathbb H\setminus\overline{\Omega},    \end{array} \right. 
\end{align} 
where 
$$\displaystyle{ K_{\theta}(\zeta -z ):= \frac{ 1}{2\pi^2
|\zeta- z|^4} [(\bar\zeta_1- \bar z_1) - i e^{-i\theta}
(\bar\zeta_2-  \bar z_2)j ]}$$
is called $\theta$-hyperholomorphic Cauchy kernel and 
$$\sigma^{\theta}_{\zeta}:=\frac{1}{2} [d \bar\zeta_{[2]} \wedge d\zeta - ie^{i\theta}d \zeta_{[1]} \wedge d\bar \zeta j],$$
where $d \zeta_{[i]}$ denotes as usual $d\zeta_0 \wedge d\zeta_1 \wedge d\zeta_2 \wedge d\zeta_3$ omitting the factor $d\zeta_i$, represents the $\mathbb H-$valued area form to $\partial \Omega$ and $d\mu_\zeta$ stands for the $4-$dimensional volume element in $\Omega$. Here  $C^1(\overline{\Omega},\mathbb H)$ denotes the subclass of functions which can be extended smoothly to an open set containing the closure $\overline{\Omega}$.
	
From the above we can see that the $\theta$-hyperholomorphic Cauchy kernel generates the following important operator.
\begin{equation}\label{e3}
 {}^\theta T [f](z):=\int_{\Omega} K_\theta(\zeta-z) f(\zeta) d\mu_\zeta
\end{equation}
on $\mathcal L_2(\Omega,\mathbb H)\cup C(\Omega,\mathbb H)$ and meets ${}^\theta\mathcal D  \circ  {}^\theta T  = I$. This can be found in \cite{GM, MS}.
  
Given a one-to-one correspondence $g :\Xi \to \Omega$  and  $h:\Omega\to \mathbb H$,  denote $W_g[ f]= f\circ g$  and $ {}^h M[ f ]=  h f$ for $f$ in a function space associated to $\Omega$. 

The M\"obius transformations in $\mathbb R^4\cong\mathbb C^2$ can be represented  by quaternionic fractional-linear transformations:
\begin{equation}\label{e4}
T(z)= (az+b)(cz+d)^{-1},
\end{equation}
where $a,b,c,d\in\mathbb H$ satisfy $ad\neq 0$  if $c=0$ {} and {} $b-ac^{-1}d\neq 0$ if $c\neq 0$, see  \cite{DF, HG}
	
Let $\Omega, \Xi\subset\mathbb R^4$ two conformal equivalent  domains, i.e., there exists $T$ given by \eqref{e4} such that $T(\Xi)=\Omega$. Denote $\zeta= T(z)\in \Omega$ for $z\in \Xi$ and define the following functions:  

\begin{align}\label{e6}
 A_{T}(z):=&   \left\{ \begin{array}{l}
\bar d, \quad  \textrm{if {} }  c=0; \\
\\
\displaystyle{\overline{(b-ac^{-1}d)} \frac{\overline{c z
(b-ac^{-1}d)^{-1} + d(b-ac^{-1}d)^{-1} }}{|c z (b-ac^{-1}d)^{-1} +
d(b-ac^{-1}d)^{-1}|^{4}} }, \quad \textrm{if {} } c\neq  0,
 \end{array}\right.  \\
    B_{T}(\zeta):=&   \left\{
\begin{array}{l}
\bar a, \quad  \textrm{if {} }  c=0; \\
\\
\displaystyle{- \bar c \overline{(\zeta -ac^{-1})}  |\zeta
-ac^{-1}|^{4} },  \quad \textrm{if {} } c\neq 0,
 \end{array}\right.\\
  C_T(z):= &  \left\{
\begin{array}{l}
 \displaystyle{\frac{|a|^2}{|d|}d^{-1},\quad \quad  \textrm{ if  $c=0$,
}}
 \\
 \\
\displaystyle{
   \frac{|c|^2}{|b-ac^{-1}d|}(b-ac^{-1}d)^{-1}
   \frac{\overline{cz(b-ac^{-1}d)^{-1}+
d(b-ac^{-1}d)^{-1}}}{|cz(b-ac^{-1}d)^{-1}+ d(b-ac^{-1}d)^{-1}|^4},  }
   \\
  \hspace{5.5cm}  \textrm{if $c\neq 0$},
\end{array} \right. \\
 \rho_T(z):= & \left\{
\begin{array}{l}
1,  \quad \quad \quad \textrm{if  $c=0$,}
\\
\\
\displaystyle{\frac{1}{|c z (b-ac^{-1}d)^{-1} +
d(b-ac^{-1}d)^{-1}|^2}}, \hspace{2cm}\textrm{ if {} } c\neq 0.
\end{array} \right. 
\end{align}

The ${\theta}$-hyperholomorphic Bergman space associated to $\Omega\subset \mathbb H$, is defined by
$${}^{\theta}\mathcal A (\Omega) :={}^{\theta}\mathfrak M(\Omega, \mathbb H)\cap \mathcal L_{2 }(\Omega,\mathbb H)$$ 
and is a quaternionic right-Hilbert space equipped with the inner product and norm inherited from $\mathcal L_{2 }(\Omega,\mathbb H)$. Moreover, there were proved the following assertions:
\begin{enumerate}
\item 	Analogue of the chain rule   
    \begin{equation} \label{e7}{}^\theta\mathcal D_z
[A_{T} f\circ T]= (B_{T}\circ T )\ {}^{\theta}
 \mathcal D_\zeta[f]\circ T,
 \end{equation}for all   $f\in C^1(\Omega, \mathbb H)$. 
\item   ${}^{C_T}M\circ W_{T}:\mathcal
L_2(\Omega,\mathbb H) \longrightarrow  \mathcal
L_{2,\rho_T}(\Xi,\mathbb H)$ is an isometric isomorphism of
quaternionic right-Hilbert spaces and 
  \begin{equation}\label{L_2ISOMOR}
  \int_{\Xi} \overline{    C_T  f\circ T   }   C_T g\circ T \rho_{_T} d\mu   =
\int_{\Omega}      \bar{f  }   g d\mu ,  
\end{equation}
for all $f,g\in\mathcal L_2(\Omega, \mathbb H)$.
\end{enumerate}  

\begin{remark}
The identity \eqref{e7} is an usual phenomena in hypercomplex analysis whose antecedent  within complex analysis is the formula used when deriving a composition of analytic functions, called chain rule and the functions $A_T$ and $B_T$ are known as  conformal weights. On the other hand, $\rho_T$  is a weight function used to define an  isometric isomorphism of quaternionic right-Hilbert spaces in \eqref{L_2ISOMOR}. For more details about conformal weights we refer the reader to \cite [pp. 364-366]{GM}.
\end{remark}
\section{$(\theta, u)$-hyperholomorphic Bergman type spaces}
Before introduce the Bergman type space to be studied, we present briefly the basic definitions and results of $(\theta, u)$-hyperholomorphic theory necessary for our purpose. For more information, we refer the reader to \cite{KS}.

Let $\Omega\subset \mathbb H$ be a domain and  $u\in \mathbb H$. The Helmholtz type operators with a quaternionic wave number $u^2$ that we are going to consider, which act on $C^2(\Omega,\mathbb H)$, must be left and right:
$$
{}_{{u^2}}\bigtriangleup=\bigtriangleup_{\mathbb R^4}+u^2 f
$$
and
$$
{}_{{u^2}, r}\bigtriangleup=\bigtriangleup_{\mathbb R^4}+ f u^2,
$$
respectively.

Consider the left- and the right- $(\theta, u)$-Cauchy-Riemann type operators defined by   
$$ {}^{{\theta}}_{{u}}\mathcal D  [f] := {}^{{\theta}}\mathcal D[f] +  u f $$
and
$$ {}^{{\theta}}\mathcal D_{{ r, u}} [f] := {}^{{\theta}}\mathcal D_r[f] + f u,$$  
for all $f\in C^1(\Omega,\mathbb H)$ respectively.

Thus
\begin{align*}
	{}^{{\theta}}\mathcal D_{{r, u}} \circ {}^{{\theta}}_{{u}}\mathcal D [f] =& 
 {}^{{\theta}}\mathcal D_r[{}^{{\theta}}_{{u}}\mathcal D [f]] + ({}^{{\theta}}_{{u}}\mathcal D [f]) u \\
= & {}^{{\theta}}\mathcal D_r[ {}^{{\theta}}\mathcal D[f] + u f] + ({}^{{\theta}}\mathcal D[f] +  u f) u \\ 
 =  &{}^{{\theta}}\mathcal D_r\circ {}^{{\theta}}\mathcal D[f] + {}^{{\theta}}\mathcal D_r[u f] + {}^{{\theta}}\mathcal D[f] u +  u f u.
\end{align*}

\begin{align*}
  \overline{{}^{{\theta}}\mathcal D}_{{r, \bar u}} \circ {}^{{\theta}}_{{u}}\mathcal D [f] =& 
 \overline{ {}^{{\theta}}\mathcal D_r}[{}^{{\theta}}_{{u}}\mathcal D  [f] ] + ({}^{{\theta}}_{{u}}\mathcal D [f] ) \bar u \\
= &   \overline{{}^{{\theta}}\mathcal D_r}\circ {}^{{\theta}}\mathcal D[f] + \overline{{}^{{\theta}}\mathcal D_r}[u f ] + {}^{{\theta}}\mathcal D[f] \bar u +  u f  \bar u.
\end{align*}
\begin{align*}
  \overline{{}^{{\theta}}_{\bar u}\mathcal D} \circ {}^{{\theta}}_{{u}}\mathcal D [f] =& 
 \overline{ {}^{{\theta}}\mathcal D}[{}^{{\theta}}_{{u}}\mathcal D  [f] ] +  \bar u ({}^{{\theta}}_{{u}}\mathcal D [f] )\\
 =   &\overline{{}^{{\theta}}\mathcal D}\circ {}^{{\theta}}\mathcal D[f] + \overline{{}^{{\theta}}\mathcal D}[u f ] +
\bar  u {}^{{\theta}}\mathcal D[f]  +  |u|^2 f\\
= & \Delta_{\mathbb R^4}[f] +  |u|^2 f  + \overline{{}^{{\theta}}\mathcal D}[u f ] + \bar u {}^{{\theta}}\mathcal D[f].
 \end{align*}
In the particular case when $f(x)$ be independent from the scalar part of $x$ and $u\in \mathbb R$ we have
\begin{align*}
\overline{ {}^{{\theta}}_{\bar u}\mathcal D} \circ {}^{{\theta}}_{{u}}\mathcal D [f] = & \Delta_{\mathbb R^4}[f] +  |u|^2 f.
\end{align*}

In analogy with the notion of $\theta-$hyperholomorphic function, consider the following definition of $(\theta, u)-$hyperholomorphic functions:

We will denote by ${}_{u}^{\theta}\mathfrak M(\Omega,\mathbb H)= C^1(\Omega,\mathbb H) \cap Ker {}^{\theta}_{u}\mathcal D$ the quaternionic right-linear space of the left $(\theta, u)-$hyperholomorphic functions, {$(\theta, u)-$hyperholomorphic functions for brevity}, associated to $\Omega \subset\mathbb H$. 

The quaternionic left-linear space of the $r-(\theta, u)-$hyperholomorphic functions on $\Omega$ will be denoted by $ {}^{\theta}\mathfrak M_{r,u}(\Omega,\mathbb H)= C^1(\Omega,\mathbb H) \cap Ker  {}^{\theta}\mathcal D_{r,u}$. 

\begin{remark}
Set $u=  u_1 + ie^{i\theta}j u_2$ for $u_1, u_2\in \mathbb C$. Then, $f=f_1+f_2j\in {}_{u }^{\theta}\mathfrak M(\Omega,\mathbb H)$ if and only if 
\begin{equation}\label{e1}
\left\{
\begin{array}{l}
\displaystyle{\frac{\partial f_1}{\partial \bar z_1} + u_1 f_1 =
ie^{i\theta}\overline{(\frac{\partial f_2
}{\partial \bar z_2}+ u_2 f_2) }},\\
\\
\displaystyle{\frac{\partial f_1}{\partial \bar z_2} + u_2 f_1 =
-ie^{i\theta}\overline{(\frac{\partial f_2}{\partial \bar z_1} + u_1f_2 ) }}.
\end{array} \right.
\end{equation}
\end{remark}

We can now state and prove the analogues of Stokes and Borel-Pompieu formulas for the $(\theta, u)$-hyperholomorphic functions: 
\begin{proposition}\label{Stokes-Borel-Pompieu}
Let $\Omega\subset \mathbb H$ be a  bounded domain with $\partial \Omega$ a $3-$dimensional compact sufficiently smooth surface. Then 
\begin{align}
d(f (z ) \nu_{u }^\theta   (z )  g(z )) = & ( {}^{\theta   }\mathcal D_{r,u }[f] (z )g (z ) + 
  f(z ) {}^{\theta   }_{u}\mathcal D[g](z ) ) e^{2  \langle u ,z  \rangle_{\theta   }}, \nonumber\\
\label{STOkes_alpha_psi}\int_{\partial \Omega} f(z ) \nu_{u }^\theta   (z )  g(z ) = &\int_{\Omega} ( {}^{\theta   }\mathcal D_{r,u }[f]( \zeta  ) g( \zeta  )  +   f( \zeta  ) {}^{\theta   }_{u }\mathcal D[g]( \zeta  ) ) d\lambda_{u }^{\theta   }( \zeta  ) 
\end{align}
and
\begin{align}  &  \int_{\partial \Omega }  ( K^{\theta   }_{u }( \zeta   -z )\sigma_{ \zeta   }^{\theta   } f( \zeta   )  +  g( \zeta   )   \sigma_{ \zeta   }^{\theta   } K^{\theta   }_{u }( \zeta   -z ) ) \nonumber  \\ 
&  - 
\int_{\Omega} (K^{\theta   }_{u } ( \zeta  -z )
  {}^{\theta   }_{u }\mathcal D [f] ( \zeta  )  +     {}^{\theta   }\mathcal D_{r, u } [g] ( \zeta  )   K^{\theta   }_{u } ( \zeta  -z )
     )d  {\mu}_ \zeta     \nonumber \\
		=  &    \label{ecuaNEWCauchw }  \left\{ \begin{array}{ll}  f(z ) + g(z ) , &  z \in \Omega,  \\ 0 , &  z \in \mathbb H\setminus\overline{\Omega} ,                    \end{array} \right. 
		\end{align} 
	for all $f,g \in C^1( {\Omega}, \mathbb H)\cap C(\overline{\Omega}, \mathbb H)$ is obtained,  
   where $
 d\lambda_{u }^{\theta   }( \zeta   )=   e^{2  \langle u , \zeta    \rangle_{\theta   }} d\mu_ \zeta   $, $ \nu_{u }^\theta   ( \zeta   ) =e^{2  \langle u , \zeta    \rangle_{\theta   }} \sigma_ \zeta   ^{\theta   }$  
and $
 K^{\theta   }_{u } ( \zeta -z ) =   e^{  \langle u ,  \zeta -z  \rangle_{\theta   }} K_{\theta   } ( \zeta -z )
 $  
is a $(\theta, u)-$hyperholomorphic reproducing function.
\end{proposition}
\begin{proof}
Note that    
\begin{align}
{}^{\theta   }\mathcal D[e^{  \langle u,z  \rangle_{\theta   }} f] 
=  &  (  y _0 e^{\langle u ,z  \rangle_{\theta   }} f +  e^{\langle y ,z  \rangle_{\theta   }} \dfrac{\partial f}
{\partial x_0}   )  +  
i(  y _1 e^{\langle u ,z  \rangle_{\theta   }} f +  e^{\langle u ,z  \rangle_{\theta   }} \dfrac{\partial f}
{\partial x _1}  )  \nonumber  \\ 
&  +
i e^{i\theta} j (  y_2 e^{\langle u ,z  \rangle_{\theta   }} f +  e^{\langle u ,z  \rangle_{\theta   }} \dfrac{\partial f}
{\partial x _2}   ) +
e^{i\theta} j (   y _3 e^{\langle u ,z  \rangle_{\theta   }} f +  e^{\langle u ,z  \rangle_{\theta   }} \dfrac{\partial f}
{\partial x _3} )  \nonumber \\
  \label{exponential0}=  &   u   e^{\langle u ,z  \rangle_{\theta   }} \ f +  e^{\langle u ,z  \rangle_{\theta   }} \ 
	{}^{\theta   }\mathcal D[f] = e^{\langle u ,z  \rangle_{\theta   }}  \  {}^{\theta   }_{u}\mathcal D[f],
\end{align}
 for all  $f \in C^1( {\Omega}, \mathbb H)$, 
where $u =   y  _0 +  y  _1 i +  y  _2 ie^{i\theta} j +  y  _3  e^{i\theta} j $  and 
$z =    x   _0 +   x   _1 i +   x   _2 ie^{i\theta} j +   x   _3  e^{i\theta} j $ with  
$ y_k , x_k \in \mathbb R$ for all $k$. 

The following identity is established by analogous reasoning   
\begin{align}
\label{exponential}{}^{\theta}\mathcal D_r[e^{  \langle u ,z \rangle_{\theta}} f] =   e^{  \langle u ,z \rangle_{\theta}} 
  \  {}^{\theta}\mathcal D_{r,u}[f],
\end{align}
for all  $f \in C^1( {\Omega}, \mathbb H)$. 

The proof is completed by applying formulas \eqref{Dif-stokes}, \eqref{Int-stokes} and \eqref{ecua4} to $e^{  \langle u ,z \rangle_{\theta}} f$, $e^{  \langle u ,z \rangle_{\theta}} g $ and making use of \eqref{exponential0} and \eqref{exponential}. \end{proof}

\begin{corollary} Cauchy-type formula for $(\theta, u)-$hyperholomorphic functions. 
\begin{align}\label{Cauchy-for-u-theta-hyper}     \int_{\partial \Omega }    K^{\theta   }_{u }( \zeta   -z )\sigma_{ \zeta   }^{\theta   } f( \zeta   )   =      \left\{ \begin{array}{ll}  f(z )   , &  z \in \Omega,  \\ 0 , &  z \in \mathbb H\setminus\overline{\Omega} ,                    \end{array} \right. 
		\end{align}
		for all $f\in 
 {}_{u}^{\theta   }\mathfrak M(\Omega) $. 
\end{corollary}
\begin{proof}
It is a direct consequence from \eqref{ecuaNEWCauchw }.  
\end{proof}

Let $f \in\mathcal L_2(\Omega,\mathbb H)\cup C(\Omega,\mathbb H)$. Consider the operator 
$${}^\theta_{u} T [f](z):=\int_{\Omega} K^{\theta   }_{u } (\zeta-z)f(\zeta) d{\mu}_\zeta.$$
	
\begin{proposition}\label{Dinver}
${}^\theta_{u}\mathcal D  \circ  {}^\theta_{u} T  = I$.
\end{proposition}
 \begin{proof}
As 
${}^\theta\mathcal D  \circ  {}^\theta T  = I $ on $\mathcal L_2(\Omega,\mathbb H)\cup C(\Omega,\mathbb H)$  then  using \eqref{exponential0} we obtain  
\begin{align*}
 e^{  \langle u ,z \rangle_{\theta   }} f(z)  = & {}^\theta\mathcal D  [  e^{  \langle u ,z \rangle_{\theta   }}  \int_{\Omega}  e^{  -\langle u , z \rangle_{\theta   }} K^{\theta   }  (\zeta-z)
   e^{  \langle u ,\zeta \rangle_{\theta   }} f  (\zeta)  d  {\mu}_\zeta  ] \\
	 = & e^{  \langle u ,z \rangle_{\theta   }}{}^\theta_{u}\mathcal D  [    \int_{\Omega}  e^{   \langle u , \zeta- z \rangle_{\theta   }} K^{\theta   }  (\zeta-z)
     f  (\zeta)  d  {\mu}_\zeta  ] ,
\end{align*}
for all $f\in \mathcal L_2(\Omega,\mathbb H)\cup C(\Omega,\mathbb H)$.
\end{proof}

\begin{remark}\label{rem8}
Let $f=f_1+f_2 j$, where $f_1,f_2: \Omega\to \mathbb C$ and define 
\begin{align*} 
{}^\theta_{u} T_1 (f_1, f_2) (z) = & \int_{\Omega} \frac{ e^{  \langle u ,z \rangle_{\theta   }}  }{   2\pi^2 |\zeta-z|^4}
 \left[    (\bar \zeta_1 - \bar z_1) f_1( \zeta) +i e^{-i\theta}  (\bar \zeta_2 - \bar z_2) \bar f_2( \zeta)  \right] d\mu_{\zeta}\\ 
& \\
{}^\theta_{u} T_2 (f_1, f_2) (z) = & \int_{\Omega} \frac{ e^{  \langle u ,z \rangle_{\theta   }}  }{   2\pi^2 |\zeta-z|^4}
 \left[    (\bar \zeta_1 - \bar z_1) f_2( \zeta) - i e^{-i\theta}  (\bar \zeta_2 - \bar z_2) \bar f_1( \zeta)  \right]d\mu_{\zeta},
\end{align*}
 where $\zeta =\zeta_1+ i e^{i\theta} j \zeta_2, \ 
 z =z_1+ i e^{i\theta} j z_2 \in \Omega$, with $\zeta_k,z_k\in\mathbb C$ for $k=1,2$. 

Therefore, ${}^\theta_{u} T  (f)  = {}^\theta_{u} T_1 (f_1, f_2) + {}^\theta_{u} T_1 (f_1, f_2) j$ for  all $f=f_1+ f_2 j \in \mathcal L_2(\Omega,\mathbb H)\cup C(\Omega,\mathbb H)$ and the previous proposition becomes 
\begin{align*}
\frac{\partial}{\partial \bar z_1} {}^\theta_{u} T_1 (f_1, f_2) - ie^{i\theta} \frac{\partial}{\partial   z_2} \overline{ {}^\theta_{u} T_2 (f_1, f_2)}   & = f_1
\\
\frac{\partial}{\partial \bar z_1} {}^\theta_{u} T_2 (f_1, f_2) + ie^{i\theta} \frac{\partial}{\partial   z_2} \overline{ {}^\theta_{u} T_1 (f_1, f_2)}   & = f_2,
\end{align*} 
on $\Omega$.
\end{remark}
	
\begin{proposition}\label{ProCon} (Chain rule or the conformal co-variant property of ${}^{{     \theta   }}_{{u}}\mathcal D$)
Given  $u , v \in\mathbb H$ and $\Omega, \Xi \subset \mathbb H$ two conformal equivalent domains; i.e., $T(\Xi)=\Omega$, where $T$ is given by \eqref{e4} and let $\zeta=T(z)\in \Omega$ for $z\in \Xi$. Then  
\begin{align*}
{}^{\theta   }_{u }\mathcal D_z   [e^{  \langle  v  - u  , z \rangle_{\theta }} A_T f\circ T ] =    
  e^{  \langle  v  - u  , z \rangle_{\theta }}   (B_T\circ T)  (   {}^{\theta }_{\delta_T} \mathcal D_{\zeta} [   f] ) \circ T   , \quad \forall f\in C^1(\Omega,\mathbb H),
\end{align*}
where the index $z$ of 
${}^{\theta   }_{u }\mathcal D_z$ denotes the quaternionic variable of differentiation, the function $\delta_T=  (B_T\circ T)^{-1}  v  (A_T\circ T^{-1} ) \quad \textrm{on} \ \Omega$, with $(B_T\circ T)^{-1}$ being the  multiplicative inverse of $B_T\circ T$ and $T^{-1}$  the inverse mapping of $T$ 
and ${}^{\theta }_{\delta_T} \mathcal D_{\zeta} [f] :=  {}^{\theta } \mathcal D_{\zeta} [f] + {\delta_T}f$. 
\end{proposition}
\begin{proof} 
From \eqref{e7} we have
\begin{align*}
  {}^{\theta   }_{u }\mathcal D_z  [e^{  \langle  v  - u  , z \rangle_{\theta }} A_T f\circ T ]  
= & u   e^{  \langle  v  - u  , z \rangle_{\theta }} A_T f\circ T    +  
{}^{\theta   }\mathcal D_z [   e^{  \langle  v - u  , z \rangle_{\theta }} A_T f\circ T  ] 
 \\
= &   
u  e^{  \langle  v  - u  , z \rangle_{\theta }} A_T f\circ T       + 
 (v  - u )e^{  \langle  v  - u  , z \rangle_{\theta }} \ A_T f\circ T \\
&  + e^{  \langle  v  - u  , z \rangle_{\theta }}  \ {}^{\theta } \mathcal D_z [ A_T f\circ T ]   \\
= &  
  e^{  \langle  v  - u  , z \rangle_{\theta }}  v  A_T f\circ T    +   e^{  \langle  v  - u  , z \rangle_{\theta }}  (B_T\circ T)  {}^{\theta }  \mathcal D_{\zeta} [   f] \circ T   \\
= &   
  e^{  \langle  v  - u  , z \rangle_{\theta }}  (B_T \circ T )  \left\{    (B_T\circ T)^{-1}  v  (A_T\circ T^{-1} )  f    + {}^{\theta } \mathcal D_\zeta [   f]  \right\} \circ T.
\end{align*} 
\end{proof}

\begin{corollary}\label{isoBergman}
Denote ${}_{\delta_T}^{ \phi} \mathfrak M(\Omega)$ the quaternionic right linear space of $f\in C^1(\Omega,\mathbb H)$ such that ${}^{\phi}\mathcal D[f] + \delta_T f = 0$ on $\Omega$. Then 
\begin{align}\label{conformCov}f\in  {}_{\delta_T}^{ \theta} \mathfrak M(\Omega)
    \   \Longleftrightarrow  \ e^{  \langle  v - u, z\rangle_{\phi}} A_T f\circ T   \in   {}_{u}^{ \theta   } \mathfrak M(\Xi) \end{align}
	 or equivalently $f=f_1+ f_2j$ we see that   
\begin{align*}
  &  \left\{
\begin{array}{l}
\displaystyle{\frac{\partial f_1}{\partial \bar \zeta  _1} + \delta_1 f_1 =
ie^{i\theta}\overline{(\frac{\partial f_2
}{\partial \bar \zeta  _2}+ \delta_2 f_2) }},\\
\\
\displaystyle{\frac{\partial f_1}{\partial \bar \zeta  _2} + \delta_2 f_1 =
-ie^{i\theta}\overline{(\frac{\partial f_2}{\partial \bar \zeta  _1} + \delta_1f_2 ) }},
\end{array} \right. 
\end{align*} 
hold on $\Omega$ if and only if  
\begin{align*} 
	\left\{
\begin{array}{l}
\displaystyle{\frac{\partial   g_1}{\partial \bar z_1} + u_1   g_1 =
ie^{i\theta}\overline{(\frac{\partial    g_2
}{\partial \bar z_2}+ u_2   g_2) }},\\
\\
\displaystyle{\frac{\partial   g_1}{\partial \bar z_2} + u_2    g_1 =
-ie^{i\theta}\overline{(\frac{\partial  
 g_2}{\partial \bar z_1} + u_1  g_2 ) }},
\end{array} \right.
\end{align*}
on $\Xi$, where $u =  u_1   + i
e^{i\theta}j u_2   $ with $u_1,u_2\in \mathbb C$, 
  $\delta_T(\zeta) =  \delta_1 (\zeta) + i
e^{i\theta}j \delta_2(\zeta)  $   
   for all $\zeta\in \Omega$
 with $\delta_1,\delta_2:\Omega\to \mathbb C$,  
and $ e^{  \langle  v-u, z  \rangle_{\theta }} A_T(z) f\circ T(z) =g_1(z)+g_2(z) j $ 
   for all $z\in \Xi$ with $g_1,g_2:\Xi\to \mathbb C$.
 \end{corollary}

\begin{proposition}\label{functionaValuation}
 Let $\Omega\subset\mathbb H$ be a bounded domain. 
 \begin{enumerate}
\item If   $f  \in {}^\theta   _{u }\mathfrak M(\Omega) $ and  $g  \in  {}^\theta\mathfrak M_{r,u }(\Omega) $ then 
 \begin{align*}  &  \int_{\partial \Omega }  ( K^{\theta   }_{u }( \zeta  -z )\sigma_{ \zeta  }^{\theta   } f( \zeta  )  +  g( \zeta  )   \sigma_{ \zeta  }^{\theta   } K^{\theta   }_{u }( \zeta  -z ) ) \nonumber  = 
		\left\{ \begin{array}{ll}  f(z ) + g(z ) , &  z \in \Omega,  \\ 0 , &  z \in \mathbb H\setminus\overline{\Omega} .                   \end{array} \right. 
\end{align*}  
\item If $z \in \Omega$ and $\epsilon>0$ are  such that $\overline{\mathbb B(z ,\epsilon)}\subset \Omega$ then there exists a constant $k_{\epsilon}>0$ such that    \begin{align*}    |f(z )| \leq    k_{\epsilon} 	\left(\int_{  \mathbb B(z ,\epsilon) } |f(\zeta  )|^2 d\mu_{\zeta}\right)^{\frac{1}{2}},
\end{align*} 
for all  $f  \in {}^\theta   _{u}\mathfrak M(\Omega) $.
\end{enumerate} 
\end{proposition}
\begin{proof}
\begin{enumerate} 
\item  Use \eqref{ecuaNEWCauchw }.   
\item  From \eqref{Cauchy-for-u-theta-hyper} and Stokes integral formula, see \eqref{STOkes_alpha_psi}, we see that 
\begin{align*}    
|f(z )|= &|\frac{ 1}{2\pi^2\epsilon^4}   \int_{  \mathbb B(z ,\epsilon) }    
    {}^\theta\mathcal D   _{r,u} [e^{-  \langle u  , \zeta +z  \rangle_{\theta   }}  (  (\bar\zeta_1- \bar z_1) - i e^{-i\theta}
(\bar\zeta_2-  \bar z_2)j ) ]  f(\zeta ) d\lambda^{\theta   }_u (\zeta )  |\\
			\leq & \frac{ 1}{2\pi^2\epsilon^4}  \int_{  \mathbb B(z ,\epsilon) }  |    
   e^{2  \langle u  , \zeta  \rangle_{\theta   }} {}^\theta\mathcal D   _{r,u } [e^{-  \langle u  , \zeta +z  \rangle_{\theta   }} ( (\bar\zeta_1- \bar z_1) - i e^{-i\theta}
(\bar\zeta_2-  \bar z_2)j ) ]  f(\zeta )  |d\mu_\zeta  ,\\
\leq  &  \frac{ 1}{2\pi^2\epsilon^4}  \left(  \int_{  \mathbb B(z ,\epsilon) }  |    
   e^{2  \langle u  , \zeta  \rangle_{\theta   }} {}^\theta\mathcal D   _{r,u } [e^{-  \langle u  , \zeta +z  \rangle_{\theta   }} ( (\bar\zeta_1- \bar z_1) - i e^{-i\theta}
(\bar\zeta_2-  \bar z_2)j ) ]    |^2 d\mu_\zeta \right)^{\frac{1}{2}}  \\ 
& \hspace{1cm}	\left(  \int_{  \mathbb B(z ,\epsilon) } |f(\zeta )|^2 d\mu_\zeta \right)^{\frac{1}{2}}.
\end{align*}
The mapping $$(\zeta_1, \zeta_2) \mapsto  e^{2  \langle u  , \zeta  \rangle_{\theta   }} {}^\theta\mathcal D   _{r,u } [e^{-  \langle u  , \zeta +z  \rangle_{\theta   }} ( (\bar\zeta_1- \bar z_1) - i e^{-i\theta}
(\bar\zeta_2-  \bar z_2)j ) ]      $$ is a continuous functions on $\mathbb C^2$ thus it is a bounded function on $\overline{\mathbb B(z ,\epsilon)}$ and there exists $M_\epsilon>0$ such that $$| e^{2  \langle u  , \zeta  \rangle_{\theta   }} {}^\theta\mathcal D   _{r,u } [e^{-  \langle u  , \zeta +z  \rangle_{\theta   }} ( (\bar\zeta_1- \bar z_1) - i e^{-i\theta}
(\bar\zeta_2-  \bar z_2)j ) ]  |  < M_{\epsilon},$$ for all $(\zeta_1, \zeta_2 ) \in \mathbb B(z ,\epsilon)$. 
Therefore,  
\begin{align*}    
|f(z )| \leq  &  \frac{ 1}{2\pi^2\epsilon^4} M_{\epsilon} \left( \frac{\pi^2 \epsilon^4}{2}   \right)^{\frac{1}{2}} 	\left(  \int_{  \mathbb B(z ,\epsilon) } |f(\zeta )|^2 d\mu_\zeta \right)^{\frac{1}{2}}.
\end{align*}
Denote $k_{\epsilon}= \dfrac{ M_{\epsilon}}{2\sqrt{2} \pi \epsilon^2}$.
\end{enumerate}

\end{proof}

\begin{definition} \label{FirstBergman}  The quaternionic right linear space 
${}^{\theta}_{u}\mathcal A(\Omega)= {}^{\theta}_{u}\mathfrak M(\Omega)  \cap \mathcal L_2(\Omega, \mathbb H)$ is called $(\theta, u)-$hyperholomorphic Bergman space associated to $\Omega$ and is denoted by  
\begin{align*}
\|f\|_{{}^{     \theta   }_{   u  }\mathcal A(\Omega)}  = \left( \int_{\Omega}|f|^2d\mu  \right)^{\frac{1}{2}},
\end{align*}
for all $f\in {}^{     \theta   }_{   u  }\mathcal A(\Omega)  $.   
\end{definition}

\begin{remark}\label{HilertSpace}
The previous sentence allows to see that the quaternionic right-linear space ${}^{\theta}_{u}\mathcal A(\Omega)$ equipped with the scalar product 
\begin{align*}
\langle  f,g \rangle_{{}^{     \theta   }_{   u   }\mathcal A(\Omega)}=
\int_{\Omega} \bar f g d\mu, \quad \forall f,g\in {}^{     \theta   }_{   u   }\mathcal A(\Omega)
\end{align*}
is a  quaternionic right-Hilbert space and from deeply similar computations to those presented in \cite{GM} and \cite{SV3} we can see that the valuation functional $f\mapsto f(z)$, for $z \in \Omega$, is bounded on $ {}^{\theta}_{u} \mathcal A(\Omega)$. 

Riesz representation theorem for quaternionic right Hilbert space, see \cite{BD},  shows that there exists 
			${}^{     \theta   }_{   u   }B_z\in {}^{     \theta   }_{  u   } \mathcal
A(\Omega)$ such that
$f(z )=\langle {}^{     \theta   }_{   u   }B_z , f \rangle_{{}^{     \theta   }_{   u   }\mathcal A(\Omega)}$ for all 
$f\in {}^{     \theta   }_{   u   }\mathcal A(\Omega)$.    Therefore  the Bergman kernel of $  {}^{     \theta   }_{    u   }\mathcal A(\Omega)$ is    
${}^{     \theta   }_{  u   }\mathcal B_\Omega(z ,\cdot)  =\overline{ {}^{     \theta   }_{   u   }B_z (\cdot)}$ and satisfies: 
  \begin{align*}f(z )=\int_{\Omega} {}^{     \theta   }_{   u   }\mathcal B_{\Omega}(z ,\zeta )f(\zeta)d\mu_{\zeta}, \quad 
\forall f\in {}^{     \theta   }_{   u   }\mathcal A(\Omega) .
\end{align*} 

The  Bergman projection associated to  ${}^{\theta}_{u}\mathcal A(\Omega) $ is         
\begin{align*}
{}^{\theta}_{u} \mathfrak B_{\Omega }[f](z ):=\int_{\Omega}
{}^{\theta}_{u}\mathcal B_{\Omega }(z ,\zeta) f(\zeta) d\mu_\zeta, \quad \forall f\in\mathcal L_{2}(\Omega,\mathbb H).
\end{align*}   
What is more, repeating the reasoning given in \cite{GM, SV3}, we obtain:  
\begin{enumerate}
\item ${}^{\theta}_{u}\mathcal B_\Omega$  is  hermitian. 
\item The mapping	$z\mapsto {}^{\theta}_{u}\mathcal B_\Omega(z, \zeta )$ belongs to 
  ${}_{u }^{  \theta   }\mathfrak M(\Omega)$ for each $\zeta \in\Omega$   and $\zeta \mapsto {}^{\theta}_{u}\mathcal B_\Omega(z, \zeta)$  belongs to   $\mathfrak M_{\bar u}^{\bar \theta}(\Omega)$  for each $z\in\Omega$. 
\item ${}^{\theta}_{u}\mathcal B_\Omega(\cdot, \cdot)$ is the unique function   {with} the previous properties. 
\item ${}^{\theta}_{u}\mathfrak B_\Omega$  is a continuous symmetric operator.
\item ${}^{\theta}_{u}\mathfrak B_\Omega[\mathcal L_2(\Omega,\mathbb H)]={}^{\theta}_{u}\mathcal A(\Omega)$.
\item $({}^{\theta}_{u}\mathfrak B_\Omega)^2={}^{\theta}_{u}\mathfrak B_\Omega$.
\end{enumerate}
\end{remark}

\begin{proposition}\label{ref123456} Given $u,v\in\mathbb H$ and
let $\Omega,\Xi\subset \mathbb H$ be two conformally equivalent  domains    { with } $T(\Xi)= \Omega $,   { where }  $T$    {  is } given by \eqref{e4}. Define   
   ${{}^{     \theta   }_{\delta_T} \mathcal A (\Omega) } :=  {}^{     \theta   }_{\delta_T} \mathfrak M(\Omega) \cap \mathcal L_{2 }(\Omega, \mathbb H)$ equipped  with the inner product inherited from the weighted space $\mathcal L_{2,\gamma_{_T}}(\Omega,\mathbb H)$,  where  $\delta_T$ is defined in Proposition \ref{ProCon}.  
	 
Denote $\gamma_{_T}(z)= e^{-2\langle  v - u , z \rangle_{\theta}} \rho_{_T}(z)$ for all $z\in \Xi$ and define 
$${}^{\theta}_{u } \mathcal A_{\gamma_{_T}} (\Xi) =  {}^{\theta}_{u} \mathfrak M(\Xi) \cap \mathcal L_{2, \gamma_{_T} }(\Xi, \mathbb H)$$ 
equipped  with the inner product inherited from the weighted space $\mathcal L_{2,\gamma_{_T}}(\Xi,\mathbb H)$, i.e.,  
\begin{align*} 
\langle f,g \rangle_{ {}^{     \theta   }_{u} \mathcal A_{\gamma_{_T}} (\Xi)} = \int_{\Xi} \bar f g  \gamma_{_T}d\mu,
\end{align*}
for all $f,g \in {}^{     \theta   }_{u } \mathcal A_{\gamma_{_T}} (\Xi)$.   Then  $ {}^{ e^{\langle v-u, z \rangle_{\theta   }}C_T } M \circ W_T \mid_{{}^{     \theta   }_{\delta_T} \mathcal A (\Omega) }:   
 {{}^{     \theta   }_{\delta_T} \mathcal A (\Omega) } \to   {}^{     \theta   }_{u} \mathcal A_{\gamma_{_T}} (\Xi) $  
    is  an isometric isomorphism of   
quaternionic right-Hilbert spaces, where $C_T$ and $\rho_{_T}$  are  given by \eqref{L_2ISOMOR}.   {What is more,}
  \begin{align}\label{equa123}  {}^\theta   _{u}\mathcal  B_{\Xi, \gamma_{_T}}(z,w )= & e^{\langle v-u, z+ w  \rangle_{\theta   }}    C_T(z) \ {}^\theta   _{\delta_T}\mathcal B_{\Omega }(T(z), T(w ))
   \overline{C_T(w )}, \\
  {}^\theta   _{u}\mathfrak B_{\Xi, \gamma_{_T}} = &  {}^{ e^{\langle v-u,  z \rangle_{\theta   }}C_T } M \circ W_T   \circ {}^\theta   _{\delta_T}\mathfrak B_{\Omega}\circ\left(   {}^{ e^{\langle v-u,  z \rangle_{\theta   }}C_T } M \circ W_T  \right)^{-1}, \nonumber 
\end{align}
   where  ${}^\theta   _{u}\mathcal  B_{\Xi, \gamma_{_T}} $ and ${}^\theta_{\delta_T}\mathcal B_{\Omega }$ are the Bergman kernels of $ {}^{     \theta   }_{u } \mathcal A_{\gamma_{_T}} (\Xi) $ and ${ {}^{     \theta   }_{\delta_T} \mathcal A  (\Omega) }$ respectively, and  $  {}^\theta   _{u}\mathfrak B_{\Xi, \gamma_{_T}}
$ and $ {}^\theta   _{\delta_T}\mathfrak B_{\Omega }$  are the Bergman projections of 
$ {}^{     \theta   }_{u } \mathcal A_{\gamma_{_T}} (\Xi) $ and ${ {}^{     \theta   }_{\delta_T} \mathcal A (\Omega) }$.
\end{proposition}
\begin{proof}  
From   Cauchy  and  Stokes formula recently showed for   $(u,\theta)-$hyperholomorphic functions on $\Omega$  we have  that  
\begin{align*}     f(z) =  &  \int_{ \partial \mathbb B(z,\epsilon)   }   K^{\theta   }_{u}(w -z )     \sigma_{w }^{\theta   } f(w )  
=  \int_{ \partial \mathbb B(z,\epsilon)   }  e^{-2  \langle u ,w   \rangle_{\theta   }}  K^{\theta   }_{u}(w -z )     \nu_u^{\theta   }(w) f(w )\\ 
=&     \frac{  1 }{2\pi^2
\epsilon^4} \int_{\partial  \mathbb B(z,\epsilon)   } e^{  \langle u  , -w-z  \rangle_{\theta   }} [(\bar w _1- \bar z_1) - i e^{-i\theta}
(\bar w _2-  \bar z_2)j ]  \nu_u^{\theta   }(w) f(w ) \\
=&     \frac{  1 }{2\pi^2
\epsilon^4} \int_{ \mathbb B(z,\epsilon)   }   
{}^{\theta}_u\mathcal D_{r}
[e^{  \langle u  , -w-z  \rangle_{\theta   }} [(\bar w _1- \bar z_1) - i e^{-i\theta}
(\bar w _2-  \bar z_2)j ] \ ]    f(w ) d\lambda_{u}^{\theta   }(w  ) \\
=&  \frac{  1 }{2\pi^2
\epsilon^4} \int_{ \mathbb B(z,\epsilon)   }  e^{2  \langle  u ,w   \rangle_{\theta   }}   
{}^{\theta}_u\mathcal D_{r}
[e^{  \langle u  , -w-z  \rangle_{\theta   }} [(\bar w _1- \bar z_1) - i e^{-i\theta}
(\bar w _2-  \bar z_2)j ] \ ]    f(w ) d\mu_w\\
=&  \frac{  1 }{2\pi^2
\epsilon^4} \int_{ \mathbb B(z,\epsilon)   }  \frac{e^{2  \langle u ,w   \rangle_{\theta   }} }{\gamma_{_T}^{\frac{1}{2}} (w)}   
{}^{\theta}_u\mathcal D_{r}
[e^{  \langle u  , -w-z  \rangle_{\theta   }} [(\bar w _1- \bar z_1) - i e^{-i\theta}
(\bar w _2-  \bar z_2)j ] \ ]    f(w ) \gamma_{_T}^{\frac{1}{2}} (w) d\mu_w
,		\end{align*}
for all  $f\in  {}^{     \theta   }_{u} \mathcal A_{\rho_{_T}} (\Xi) $, where  
 $\mathbb B(z,\epsilon) \subset\Omega$, and the Cauchy-Schwarz inequality gives us that 
\begin{align*}    | f(z) | \leq &   \frac{  1 }{2\pi^2
\epsilon^4} \int_{ \mathbb B(z,\epsilon)   }  \mid  \frac{e^{2  \langle u ,w   \rangle_{\theta   }} }{
\sqrt{\gamma_{_T}  (w) }}   
{}^{\theta}_u\mathcal D_{r}
[e^{  \langle u  , -w-z  \rangle_{\theta   }} [(\bar w _1- \bar z_1) - i e^{-i\theta}
(\bar w _2-  \bar z_2)j ] \ ]    f(w ) \sqrt{\gamma_{_T}  (w) }\mid   d\mu_w \\
\leq &   \frac{  1 }{2\pi^2
\epsilon^4} 
\left( \int_{ \mathbb B(z,\epsilon)   } | \ \frac{e^{2  \langle u ,w   \rangle_{\theta   }} }{\sqrt{\gamma_{_T}  (w) }} 
{}^{\theta}_u\mathcal D_{r}
[e^{  \langle u  , -w-z  \rangle_{\theta   }} [(\bar w _1- \bar z_1) - i e^{-i\theta}
(\bar w _2-  \bar z_2)j ] \ ] \ |^2 d\mu_w \right)^{\frac{1}{2} }  \\ 
&  \hspace{2cm} 
\left( \int_{ \mathbb B(z,\epsilon)   } | f(w )|^2  \gamma_{_T}  (w)    d\mu_w \right)^{\frac{1}{2} },
		\end{align*}
for all  $f\in  {}^{     \theta   }_{u} \mathcal A_{ \gamma_{_T}} (\Xi) $. Therefore, there exists  $K_{\epsilon}>0$, that depends of  $\epsilon>0$ such that 
$| f(z) |   
\leq K_{\epsilon} \| f \|_{{}^{     \theta   }_{u} \mathcal A_{\gamma_{_T}} (\Xi) }$
. Using   the previous inequality we obtain   that   
 ${}^{     \theta   }_{u } \mathcal A_{ \gamma_{_T}} (\Xi)$ is a quaternionic right-Hilbert space  whose valuation functional is bounded and as a consequence   ${}^{     \theta   }_{u} \mathcal A_{ \gamma_{_T}} (\Xi)$ has a reproducing  kernel  and a projection.

Due to   $C_T=\lambda_T A_T$  where $\lambda_T $ is a positive real constant,  Corollary \ref{isoBergman} gives us   
$    e^{\langle v-u, \cdot  \rangle_{\theta   }}C_T   f\circ T  \in  
{}^{     \theta   }_{u} \mathfrak M(\Xi) $ iff  $ f \in {}^{     \theta   }_{\delta_{T} } \mathfrak M (\Omega)$ and   
\begin{align*}
& \int_{\Xi} \overline{   e^{\langle v-u, z \rangle_{\theta   }}C_T(z) f\circ T(z)  } 
    e^{\langle v-u, z \rangle_{\theta   }}C_T(z) g\circ T(z)   e^{-2\langle v-u, z \rangle_{\theta   }} \rho_{_T}(z) d\mu_z  \\ = &
\int_{\Omega}      \overline{f(\zeta) }   g(\zeta)  d\mu_{\zeta},
\end{align*}
obtained from  \eqref{L_2ISOMOR}, we conclude that  
 $$
   \langle{}^{ e^{\langle  v-u , z \rangle_{\theta   }}C_T } M \circ W_T [ f], {}^{ e^{\langle   v-u , z \rangle_{\theta   }}C_T } M \circ W_T[ g]\rangle_{
{}^\theta   _{u} \mathcal A_{\gamma_{_T}}(\Xi)} 
=  \langle   f, g  \rangle_{
{}^\theta   _{\delta_T} \mathcal A(\Omega)},$$  
for all $f,g\in {}^\theta   _{\delta_{T}} \mathcal A(\Omega)$. Therefore, ${}^{ e^{\langle  v-u , z \rangle_{\theta   }}C_T } M \circ W_T$ can be used for to show  that  
  ${}^\theta   _{\delta_{T}} \mathcal A (\Omega)$ has the same properties that  $
{}^\theta   _{u} \mathcal A_{\gamma_{_T}}(\Xi)$, i.e., ${}^\theta   _{\delta_{T}} \mathcal A (\Omega)$ is a quaternionic right-Hilbert space  with  a reproducing  kernel  and a projection.

Finally, if  $h \in   { {}^\theta   _{u} \mathcal A_{\gamma_{_T}} (\Xi)} $ then  
 $\left( {}^{ e^{\langle  v-u , z \rangle_{\theta   }}C_T } M \circ W_T  \right)^{-1}[h] \in {}^\theta   _{\delta_{T}} \mathcal A(\Omega)$
 and    
\begin{align*}
\left( {}^{ e^{\langle   v-u , z \rangle_{\theta   }}C_T } M \circ W_T  \right)^{-1}[h] (\zeta) = 
 & \int_{\Omega}   {}^{\theta   }_{\delta_T} \mathcal B_{\Omega,{\gamma_{_T}}}( \zeta, z ) 
 \left( {}^{ e^{\langle   v-u , z \rangle_{\theta   }}C_T } M \circ W_T  \right)^{-1}[h] (z )   d\mu_{z}  \\
= &  \langle    {}^{\theta   }_{\delta_T} \mathcal B_{\Omega}(\cdot,\zeta)  ,   \left( {}^{ e^{\langle   v-u , z \rangle_{\theta   }}C_T } M \circ W_T  \right)^{-1}[h]  \rangle_{
{}^\theta   _{\delta_T} \mathcal A(\Omega)} \\
 = & \langle  {}^{ e^{\langle   v-u , z \rangle_{\theta   }}C_T } M \circ W_T[ {}^{\theta   }_{\delta_T} \mathcal B_{\Omega,{\gamma_{_T}}}(\cdot,\zeta)]  ,   h   \rangle_{
{}^\theta   _{ u } \mathcal A_{   { \gamma_{_T}   }  }(\Xi)} \\
= &   \int_{\Xi} \overline{ {}^{ e^{\langle   v-u , z \rangle_{\theta   }}C_T } M \circ W_T [ {}^{\theta   }_{\delta_T} \mathcal B_{\Omega,{\gamma_{_T}}}(\cdot,\zeta)] }  \     h 
   { \gamma_{_T} } d\mu .
\end{align*}
Apply    ${}^{ e^{\langle  v-u, z \rangle_{\theta   }}C_T } M \circ W_T$ for to see that   
\begin{align*}
h (z) =     \int_{\Xi}   e^{\langle   v-u ,  z + w \rangle_{\theta   }} C_T(z)  \    
 {}^{\theta   }_{\delta_T} \mathcal B_{\Omega}(T(z), T(w)) \overline{ C_T(w)}     h(w)    { 
\gamma_{_T} (w)}  d\mu_{w} .  
\end{align*}
The relationships between the kernels and the projections are consequences of the   identity given above. 
\end{proof}
There exists another kind of  weighted quaternionic Bergman spaces induced by $Ker({}_{u }^{\theta}\mathcal D )$.
\begin{definition}\label{lastBergman}
Given a domain $\Omega\subset\mathbb H$ define   
$${}_{{\lambda}_{u}^\theta   }\mathcal A  (\Omega)=\{ f\in  {}^\theta   _{u}\mathfrak M(\Omega) \ \mid \ 
		\int_{\Omega} |f(z)|^2 d\lambda^{\theta   }_{u}(z)     <\infty\}.  
		$$   
\begin{align*}
 		\|f\|_{{}_{{\lambda}_{ u }^\theta   }\mathcal A  (\Omega)}^2: = &   \int_{\Omega} |f|^2 d\lambda^{\theta   }_{ u }  , \quad  
	\left\langle f,g\right\rangle_{{}_{{\lambda}_{ u }^\theta   }\mathcal A  (\Omega)}:=  
	\int_{\Omega} \bar f g d\lambda^{\theta   }_{ u }, \\
		{}_{ u }^\theta   \mathcal S [f](z):= &  e^{\left\langle u , z \right\rangle_{\theta   }} f(z) , \quad    \forall  z\in \Omega,  
\end{align*}
for all   $f,g\in {}_{{\lambda}_{u}^\theta}\mathcal A (\Omega)$. Let us recall that $d\lambda_{u }^{\theta   }( \zeta   )=   e^{2  \langle u , \zeta    \rangle_{\theta}} d\mu_ \zeta$.
\end{definition}    
We abbreviate ${}_{\lambda_{0}^\theta}\mathcal A(\Omega)$ to ${}_{0}^\theta \mathcal A(\Omega)$ the quaternionic right-Hilbert space, written ${}^\theta \mathcal A(\Omega)$ in \cite{GM}. 

\begin{proposition}\label{proo2}   Given $ u,v \in \mathbb H$ then  ${}_{u}^\theta   \mathcal S :{}_{   \lambda   _{ u }^\theta   } \mathcal A(\Omega)\to  {}^\theta   \mathcal A(\Omega) $ and 
${}_{ u-v  }^\theta     \mathcal S  : {}_{   \lambda   _{ u }^\theta} \mathcal A(\Omega) \to {}_{   \lambda   _{ v }^\theta   } \mathcal A(\Omega)$  
are isometric isomorphisms of quaternionic right-Hilbert spaces and  
\begin{align*} {}_{   \lambda   _{ v }^\theta   }\mathcal  B_{\Omega}(z ,\zeta)= &  e^{ \langle  u - v ,  z  +\zeta   \rangle_{\theta   }} {}_{ \lambda   _{ u }^\theta   }\mathcal B_{\Omega}( z ,\zeta ) ,\\
 {}_{   \lambda   _{ v }^\theta   }\mathfrak B_{\Omega} =&  {}_{ u-v  }^\theta       \mathcal S  \circ  {}_{   \lambda   _{ u }^\theta   }     \mathfrak B_{\Omega}\circ {}_{ v-u  }^\theta    \mathcal S,
\end{align*}
  { where  
${}_{   \lambda   _{ v }^\theta   }\mathcal  B_{\Omega}$ and  $ {}_{   \lambda   _{ v }^\theta   }\mathfrak B_{\Omega}$ 
are  the kernel and the projection 
 of  $ {}_{   \lambda   _{ v }^\theta   } \mathcal A(\Omega)$ respectively and
 $  {}_{ \lambda   _{ u }^\theta   }\mathcal B_{\Omega}$ and $ {}_{   \lambda   _{ u }^\theta   }     \mathfrak B_{\Omega}$ 
are  the kernel and the projection  of ${}_{   \lambda   _{ u }^\theta   } \mathcal A(\Omega)$.}
\end{proposition}
\begin{proof}  { If $f,g\in {}_{   \lambda   _{ u }^\theta   } \mathcal A(\Omega)$ 
   then ${}_{ u }^\theta   \mathcal S[f], {}_{ u }^\theta   \mathcal S[g]\in {}^\theta   \mathfrak M(\Omega)$, see  \eqref{exponential}, } and  
\begin{align}
\langle  {}_{ u }^\theta   \mathcal S[f] , {}_{ u }^\theta   \mathcal S[g] \rangle_{ {}^\theta   \mathcal A(\Omega)} = & \int_{\Omega} \overline{{}_{ u }^\theta   \mathcal S[f]} {}_{ u }^\theta   \mathcal S[g] d\mu 
= \int_{\Omega} \overline{ f(x) }  g(x) e^{2\langle  u , x  \rangle_{\theta   }} d\mu_x \nonumber \\
= & \label{finalprop} \langle  f,g \rangle_{ {}_{   \lambda   _{ u }^\theta   } \mathcal A(\Omega)} .
\end{align}
  As ${}_{ u }^\theta   \mathcal S :{}_{   \lambda   _{ u }^\theta   } \mathcal A(\Omega) \to   {}^\theta \mathcal A(\Omega)$ is a bijective quaternionic right-linear  operator that preserves the inner product we conclude that   ${}_{\lambda_{u}^\theta} \mathcal A(\Omega) $ is a copy of $ {}^\theta  \mathcal A(\Omega)$, as quaternionic right-Hilbert spaces, and  \eqref{finalprop} is used to  prove the relationships between the  reproducing kernels and  between the projections.  For the rest of conclusions note that 
$$({}_{v}^\theta   \mathcal S)^{-1}= {}_{-v}^\theta \mathcal S$$ 
and 
$${}_{u-v}^\theta \mathcal S = {}_{ -v }^\theta \mathcal S \circ {}_{u}^\theta \mathcal S$$.     
\end{proof}

\begin{remark}
\begin{enumerate}
\item If there exists $\lambda \in \mathbb R$ such that $\Omega \subset \{ z \in \mathbb H \  \mid \  \left\langle u , z \right\rangle_{\theta   } <\lambda\}$. Then $  {}_{ u }^\theta   \mathcal A (\Omega)\subset {}_{   \lambda   _{ u }^\theta   }\mathcal A (\Omega) $ as  function sets.
\item If there exists $\lambda \in \mathbb R$ such that 
$\Omega \subset \{ z \in \mathbb H \  \mid \  \left\langle u , z \right\rangle_{\theta   } > \lambda   \}$. Then
${}_{   \lambda   _{ u }^\theta   }\mathcal A (\Omega)  \subset   {}_{ u }^\theta   \mathcal A (\Omega) $  as  function sets. 
\item If $\Omega$ is a bounded domain hence $  {}_{   \lambda   _{ u }^\theta   }\mathcal A (\Omega)  =  {}_{ u }^\theta    \mathcal A (\Omega) $ as  function sets.
\end{enumerate}
\end{remark}

\section{The $(\theta, u)-$hyperholomorphic Bergman type spaces in domains of $\mathbb C^2$}
To illustrate the main ideas and motivation of this paper, we develop a theory of Bergman spaces for certain family of holomorphic function on the setting of bounded smooth domains in $\mathbb C^2$, equipped with the topology induced by its usual norm.
\begin{definition}	Let $\Omega\subset \mathbb C^2\cong \mathbb H$ be a domain and set $\alpha,\beta\in \mathbb C$. The  real linear space of $(\alpha, \beta)-$holomorphic functions on $\Omega$, denoted by ${}_{\alpha,\beta }Hol(\Omega, \mathbb C)$, consists of all $f \in C^1(\Omega,\mathbb C)$ such that   
\begin{align}\label{VSistem}
\begin{array}{l}
\displaystyle\frac{\partial f}{\partial \bar z_1} =- \alpha  f \\
\\
\displaystyle\frac{\partial f}{\partial \bar z_2} = - \beta f .
\end{array}  
\end{align}
\end{definition}
This definition derives from the identity 
$$ {}^{\theta}\mathcal D_{r,(\alpha +ie^{i\theta} j  \beta)}[f]=0$$ 
with $f \in C^1(\Omega,\mathbb C)$.

Note that solutions of (\ref{VSistem}) can be thought as $(\theta, u)-$hyperholomorphic functions, when $u=\alpha +ie^{i\theta} j\beta$ using (\ref{e1}).

It is the fact that ${}_{0,0 }Hol(\Omega, \mathbb C) = Hol(\Omega, \mathbb C)$ that makes notation ${}_{\alpha,\beta }Hol(\Omega, \mathbb C)$ allowable for a natural generalization of the space of holomorphic functions.

Application of $ {}^{\theta}\mathcal D_{r,(\alpha +ie^{i\theta} j  \beta)}$ to $g \in  C^{1}(\Omega,\mathbb C)$ gives 
$${}^{\theta}\mathcal D_{r,(\alpha+ie^{i\theta} j  \beta )} [g] =  \displaystyle\frac{\partial g}{\partial \bar z_1} +  \alpha   g  + 
i e^{i\theta} j\left( \displaystyle\frac{\partial \bar g}{\partial \bar z_2}+  \beta  \bar g   \right)$$ 
and hence $g \in Ker \ {}^{\theta}\mathcal D_{r,(\alpha+ie^{i\theta} j  \beta)} \cap C^{1}(\Omega,\mathbb C)$ if and only if
\begin{align}\label{antiVSistem}
\begin{array}{l}
\displaystyle\displaystyle\frac{\partial g}{\partial \bar z_1} = -  \alpha   g  \\
\\
\displaystyle\frac{\partial g}{\partial z_2}= - \bar \beta g.
\end{array}  
\end{align}
Unfortunately, here appears implicitly a condition of non-homogeneous anti-holomorphy.   

Now we shall consider Stokes and Borel-Pompieu types formulas induced by equations \eqref{VSistem} and \eqref{antiVSistem}. Also  Cauchy type  formula  for $ {}_{\alpha,\beta}Hol(\Omega, \mathbb C)$ and a  kind inverse of operator $\displaystyle (\frac{\partial  }{\partial \bar z_1}, \displaystyle\frac{\partial}{\partial z_2} )$ are presented.   
\begin{corollary} 
Let $\Omega\subset \mathbb C^2$ be a bounded domain with $\partial \Omega$ a 3-dimensional compact sufficiently smooth surface. 
\begin{enumerate}
\item  Integral Stokes formula
\begin{align*}
\int_{\partial \Omega} f \nu_1   g = &\int_{\Omega} [ \ (
\frac{\partial f} {\partial \bar z_1}(z) + \alpha   f(z) )  g(z) + f(z)  (\frac{\partial g} {\partial \bar z_1}(z)  + \alpha  g(z) )\ ]
  e^{ \textrm{Re} (\alpha z_1+ \beta z_2)} d\mu_z 
\end{align*}
and   
\begin{align*}
\int_{\partial \Omega} \bar{f }\nu_2  g = &\int_{\Omega} [ \ 
\overline{(\frac{\partial f} {\partial \bar z_2}(z)  + \beta f(z) )} g(z) + \overline{f(z)} ( \frac{\partial\bar g} {\partial \bar z_2}(z) + \beta  \overline{ g(z) }) \ ]  e^{ \textrm{Re} (\bar \alpha z_1+ \bar \beta z_2)} d\mu_z, 
\end{align*}
	for all $f,g \in  C^1( {\Omega}, \mathbb C)\cap C (\overline{\Omega}, \mathbb C)$,   where 
$	z= z_1 + i
e^{i\theta}j z_2  \in \Omega$.  	The complex differential forms $\nu_1$ and $\nu_2$ are  such that 
$  \nu_{  \alpha +i
e^{i\theta}j \beta }^\theta  =
\nu_1 + i
e^{i\theta}j\nu_2 $.
\item Borel-Pompieu Formula.
\begin{align*}  &  \int_{\partial \Omega} (K_1(\zeta   -z )\sigma_1(\zeta )  +     K_2(\zeta   -z ) \sigma_2(\zeta ) ) f(\zeta   )  \\
&   - 
\int_{\Omega} \left(   \ K_1(\zeta  -z ) (\frac{\partial f}{\partial \bar \zeta _1}(\zeta ) +  \alpha   f(\zeta )) +   K_2(\zeta   -z ) ( \frac{\partial f}{\partial \bar \zeta _2}(\zeta ) +   \beta f(\zeta ) )  \right)d  {\mu}_\zeta     \nonumber \\
		=  &      \left\{ \begin{array}{ll}  f(z ) , &  z \in \Omega,  \\ 0 , &  z \in \mathbb H\setminus\overline{\Omega} ,                    \end{array} \right. 
		\end{align*}  
 	for all $f \in  C^1( {\Omega}, \mathbb C)\cap C (\overline{\Omega}, \mathbb C)$,  where 
\begin{align*} 
K_1( \zeta   -z )= & \frac{ e^{    \textrm{Re} ( \bar  \alpha   (\zeta_1-z_1)  + \bar  \beta     (\zeta_2 -z_2) )  }}{2\pi^2
| \zeta  - z|^4}  (\bar \zeta  _1- \bar z_1) , \\
  K_2( \zeta   -z )= & \frac{ e^{    \textrm{Re} ( \bar  \alpha  ( \zeta_1 - z_1)  + \bar  \beta  (\zeta_2 -  z_2) )  -2i\theta } 
 }{2\pi^2
| \zeta  - z|^4} 
( \bar  \zeta  _2- \bar z_2)  , 
\end{align*}
 and $ \sigma_1$, $\sigma_2 $ are complex differential forms such that $\sigma_{\zeta}^{\theta} = \sigma_1(\zeta )+ie^{i\theta} j \sigma_2(\zeta )$ and $\zeta= \zeta_1 + ie^{i\theta}j \zeta_2$
\item Cauchy type formula in $ {}_{\alpha,\beta}Hol(\Omega, \mathbb C)$.
\begin{align*}    
\int_{\partial \Omega } ( \ K_1(\zeta   -z )\sigma_1 (\zeta ) +     K_2(\zeta   -z ) \sigma_2(\zeta ) \ ) f(\zeta   )  		=  &      \left\{ \begin{array}{ll}  f(z ) , &  z \in \Omega,  \\ 0 , &  z \in \mathbb H\setminus\overline{\Omega}, \end{array} \right. 
\end{align*} 
for $f \in  {}_{\alpha,\beta }Hol(\Omega) $. Hence, the pair	$(K_1 , K_2 )$ are our reproducing functions.	
\item (A kind of inverse  operators of $\displaystyle (\frac{\partial  }{\partial \bar z_1}, \displaystyle\frac{\partial}{\partial z_2})$). 
For any $f \in \mathcal L_2(\Omega,\mathbb C)\cup C(\Omega,\mathbb C)$ there holds that 
\begin{align}\label{systeminv}   
\frac{\partial}{\partial \bar z_1} {}_{\alpha,\beta} T_1 (f) +  \frac{\partial}{\partial   z_2} \overline{ {}_{\alpha,\beta} T_2 (f)}   & = f 
\\
 -  \frac{\partial}{\partial \bar z_1} {}_{\alpha,\beta} T_2 (f ) +   \frac{\partial}{\partial   z_2} \overline{ {}_{\alpha,\beta} T_1 (f)}   & = 0,
\end{align} 
on $\Omega$, where 
\begin{align*} 
{}_{\alpha,\beta} T_1 (f) (z) = & \int_{\Omega} 
K_1( \zeta   -z )
f ( \zeta) 
  d\mu_{\zeta} \\ 
& \\
{}_{\alpha,\beta} T_2 (f) (z) = & \int_{\Omega} K_2( \zeta   -z ) \bar f ( \zeta) 
 d\mu_{\zeta} ,
\end{align*}
 where $\zeta =\zeta_1+ i e^{i\theta} j \zeta_2, \ 
 z =z_1+ i e^{i\theta} j z_2 \in \Omega$, with $\zeta_k,z_k\in\mathbb C$ for $k=1,2$. 
\end{enumerate}
\end{corollary}
\begin{proof} 
Items $1.$, $2.$ and $3.$ are consequences of Proposition \ref{Stokes-Borel-Pompieu} and making use of
\begin{align*} 
  \langle  \alpha +i
e^{i\theta}j \beta  ,  z_1 +i
e^{i\theta}j z_2    \rangle_{\theta   } = &  (\frac{\alpha  +\bar \alpha }{2} )(\frac{z_1 +\bar z_1}{2}) +( 
		 \frac{\alpha  -\bar \alpha }{2i})( \frac{z_1 -\bar z_1}{2i} ) \\
		 &+ (\frac{\beta  +\bar \beta }{2})( \frac{z_2 +\bar z_2}{2} )+ (
		 \frac{\beta  -\bar \beta }{2i})( \frac{z_2 -\bar z_2}{2i}) \\
			= &  \frac{1}{2} (  \alpha  \bar z_1  +   \bar \alpha     z_1  +   \beta   \bar z_2 +  \bar \beta  z_2) \\
						= &    \textrm{Re} ( \bar  \alpha   z_1  + \bar  \beta     z_2)  
\end{align*}
 and 
\begin{align*} 
d\lambda_{\alpha+ i e^{i\theta} j \beta }^{\theta}(z ) = &  e^{\textrm{Re} ( \bar  \alpha z_1  + \bar \beta z_2) } d\mu_z.
\end{align*}

Finally, item $4.$ follows from Proposition \ref{Dinver} and remark \ref{rem8}.
\end{proof}
\begin{remark}
Equation \eqref{systeminv} can be rewritten as follows:
\begin{align*} \left(
\begin{array} {cc} 
\displaystyle 
\frac{\partial}{\partial \bar z_1} &  \displaystyle  \frac{\partial}{\partial   z_2} \\ 
 &\\
 \displaystyle  \frac{\partial}{\partial  \bar z_2}  & \displaystyle  -  \frac{\partial}{ \partial z_1} 
\end{array}\right) 
\left(
\begin{array} {c} 
\displaystyle 
{}_{\alpha,\beta} T_1 (f) \\
\\
\displaystyle  \overline{ {}_{\alpha,\beta} T_2 (f)} \end{array}\right) 
=
\left(
\begin{array} {c} 
\displaystyle 
f\\
\\
 0 \end{array}\right) 
\end{align*} 
on $\Omega$, for any $f \in \mathcal L_2(\Omega,\mathbb C)\cup C(\Omega,\mathbb C)$.
\end{remark}

\begin{corollary}  Consider  $\alpha,\beta , \chi, \xi \in\mathbb C$ and  $\Omega, \Xi \subset \mathbb C^2 \approx \mathbb H$  such that    
$T(\Xi)=\Omega$, where $T$ is given by \eqref{e4} for $c=0$ and $d\in\mathbb C$. 

Denote $\zeta_1+i e^{i\theta } j \zeta_2  =T(z_1+ i e^{i\theta } jz_2  ) \in \Omega $ for $z_1+ i e^{i\theta } jz_2\in \Xi $. Then  
  $f \in C^1(\Omega,\mathbb C)$ satisfies 
\begin{align*} \left\{
\begin{array}{l}
\displaystyle\frac{\partial f}{\partial \bar \zeta_1} =- \delta_1  f \\
\\
\displaystyle\frac{\partial f}{\partial \bar \zeta_2} = - \delta_2 f .
\end{array} \right.
\end{align*}
on $\Omega$ if and only if 	
\begin{align*}
\left\{
\begin{array}{l}
\displaystyle\frac{\partial ( \
e^{ \textrm{Re} ( \overline{( \chi - \alpha  )} z_1  + \overline{( \xi - \beta  )}    z_2) } 
 A_T f\circ T   ) }{\partial \bar z_1} =- \alpha   e^{ \textrm{Re} ( \overline{( \chi - \alpha  )} z_1  + \overline{( \xi - \beta  )}    z_2) } A_T f\circ T  \\
\\
\displaystyle\frac{\partial  ( \ e^{ \textrm{Re} ( \overline{( \chi - \alpha  )} z_1  + \overline{( \xi - \beta  )}    z_2) } A_T f\circ T ) }{\partial \bar z_2} = - \beta  
e^{ \textrm{Re} ( \overline{( \chi - \alpha  )} z_1  + \overline{( \xi - \beta  )}    z_2) } A_T f\circ T .
\end{array} \right.
\end{align*}
on $\Xi$, where 
$\delta_T=  (B_T\circ T)^{-1}  (\chi + i e^{i\theta} j \xi )
  (A_T\circ T^{-1} )  =  \delta_1 + ie^{i\theta} \delta_2$ and $ \delta_1, \delta_2:\Omega\to \mathbb C$.
\end{corollary}
\begin{proof} We have only to use the fact that conditions on $T$ imply $A_T$ is a complex valued function and application of Proposition \ref{ProCon}.
\end{proof}

We will present some properties of the Bergman space induced by \eqref{VSistem}.
\begin{definition}\label{ComplexBergman}  
The complex  linear space 
${ }_{\alpha,\beta}A(\Omega)= { }_{\alpha,\beta}Hol(\Omega, \mathbb C)  \cap \mathcal L_2(\Omega, \mathbb C)$ is called  
 $ \alpha,\beta$-holomorphic Bergman space associated to $\Omega$ and denote  
\begin{align*}
\|f\|_{{ }_{   \alpha,\beta}A(\Omega)}  = & \left( \int_{\Omega}|f|^2d\mu  \right)^{\frac{1}{2}},
\langle f, g\rangle_{{ }_{ \alpha,\beta  }A(\Omega)}  = \int_{\Omega}\bar f g d\mu  ,
\end{align*}
for all $f,g\in  { }_{\alpha,\beta }A(\Omega)$.   
\end{definition}
\begin{remark}  
Due to ${ }_{\alpha,\beta}A(\Omega)$ is a closed  subset of ${}_{\alpha+ i e^{i\theta}\beta}{}^{\theta}\mathcal A(\Omega)$ and that  
$$\langle f, g\rangle_{{ }_{   \alpha,\beta  }A(\Omega)} = \langle f, g\rangle_{{}_{\alpha+ i e^{i\theta}\beta}{}^{\theta}\mathcal A(\Omega)}$$ 
for all $f, g\in  { }_{\alpha,\beta}A(\Omega)$ we obtain that $( { }_{\alpha,\beta}A(\Omega), \langle \cdot, \cdot \rangle_{{ }_{\alpha, \beta}A(\Omega)})$ is a complex Hilbert space. 

As the valuation functional is bounded on  ${}_{\alpha+ i e^{i\theta}\beta}{}^{\theta}\mathcal A(\Omega)$ so it is bounded on ${ }_{\alpha, \beta} A(\Omega)$. Therefore, Riesz representation theorem for complex linear spaces establishes the existence of a reproducing  kernel  ${ }_{\alpha,\beta}B_{\Omega}:\Omega\times \Omega \to \mathbb C$ such that
$$f(z) =\int_{\Omega} { }_{   \alpha ,\beta  }B_{\Omega}(z,w) f(w)d\mu_w, \quad \forall f\in { }_{\alpha, \beta}A(\Omega)$$ 
and the Bergman projection associated to  ${ }_{\alpha, \beta} A(\Omega)$ is given by         
$$
{ }_{\alpha,\beta} { \bf B}_{\Omega }[f](z ) =\int_{\Omega}
{  }_{   \alpha,\beta   }  B_{\Omega }(z ,\zeta) f(\zeta) d\mu_\zeta, \quad \forall f\in\mathcal L_{2}(\Omega,\mathbb C).
$$
The space ${ }_{  0, 0  }A(\Omega)$ was studied in \cite{GM}. Thus this paper is an  extension, preserving the structure, of \cite{GM}.

The following properties of  ${ }_{\alpha, \beta} B_{\Omega }$ and ${ }_{\alpha, \beta} {\bf B}_{\Omega}$ can be directly verified. This follows by the same method as in \cite{GM}.
\begin{enumerate}
\item ${  }_{\alpha,\beta}  B_{\Omega }$   is  hermitian in complex sense. 
\item The mapping $z\mapsto { }_{\alpha,\beta} B_\Omega(z, \zeta)$ belongs to ${}_{\alpha,\beta }Hol(\Omega,\mathbb C)$ for each $\zeta \in\Omega$ fix and $h(\zeta) = {}_ {\alpha} B_\Omega(z, \zeta )$ satisfies 
$$\begin{array}{l}
\displaystyle\frac{\partial h}{\partial   \zeta_1} =- \bar\alpha  h \\
\\
\displaystyle\frac{\partial h}{\partial \zeta_2} = - \bar\beta h ,
\end{array}$$
on $\Omega$  for each $z\in \Omega$.
\item ${  }_{\alpha,\beta} B_\Omega(\cdot, \cdot)$ is the unique reproducing kernel holding  the previous properties. 	
\item ${  }_{\alpha,\beta} {\bf B}_{\Omega }$ is a continuous symmetric operator such that 
$${ }_{\alpha,\beta} {\bf B}_{\Omega }[\mathcal L_2(\Omega,\mathbb C)]={  }_{\alpha,\beta} A(\Omega),$$ 
$$({  }_{\alpha,\beta} {\bf B}_{\Omega })^2={ }_{\alpha,\beta} {\bf B}_{\Omega}.$$
\end{enumerate}
\end{remark}

\begin{corollary}
Let $\Omega, \Xi\subset \mathbb C^2\approx \mathbb H$ such that  $T(\Xi)= \Omega $,   { where }  $T$    {  is } given by \eqref{e4}  such that   and $c=0$ and $d\in \mathbb C$. Given   $\alpha,\beta, \chi , \xi \in \mathbb C$.  
Consider the following Bergman type space     
$$
{}_{\delta_T} A(\Omega) = {}_{\delta_1,\delta_2}Hol(\Omega,\mathbb C)\cap \mathcal L_2(\Omega, \mathbb C)
$$ 
and 
$$ {}_{ \alpha,  \beta  }  A_{\gamma_{_T}} (\Xi)  =  {}_{ \alpha, \beta}Hol(\Xi,\mathbb C) \cap \mathcal L_{2, \gamma_{_T} }(\Xi, \mathbb C).$$ 
Recall that  
$$\delta_T=(B_T\circ T)^{-1} (\chi + i e^{i\theta} j \xi ) (A_T\circ T^{-1} ) = \delta_1 + ie^{i\theta} \delta_2,$$
$$\gamma_{_T}(z)= e^{- { \textrm{Re} ( \overline{( \chi - \alpha  )} z_1  + \overline{( \xi - \beta)} z_2) }} \rho_{_T}(z).$$ 
Then  
$$J_T: {}^{ e^{ \textrm{Re} ( \overline{( \chi - \alpha  )} (z_1+\zeta_1)  + \overline{( \xi - \beta  )}    (z_2+\zeta_2) ) } C_T } M \circ W_T \mid_{{}
	_{\delta_1,\delta_2}   A (\Omega) }:   
 {{}
	_{\delta_1,\delta_2}  A (\Omega) } \to   {  }_{ \alpha,  \beta  }  A_{\gamma_{_T}} (\Xi) $$  
is an isometric isomorphism of complex Hilbert spaces, where $C_T, \rho_{_T}$  are  given by \eqref{L_2ISOMOR}. 

In addition, denoting $\zeta_1+i e^{i\theta } j \zeta_2  =T(z_1+ i e^{i\theta } jz_2  ) \in \Omega $ for $z_1+ i e^{i\theta } jz_2\in \Xi $ we have that  
  \begin{align*}  {} _{\alpha,\beta}  B_{\Xi, \gamma_{_T}}(z,\zeta  )= & e^{ \textrm{Re} ( \overline{( \chi - \alpha  )} (z_1+\zeta_1)  + \overline{( \xi - \beta  )}    (z_2+\zeta_2) ) }    C_T(z) \ {}_{\delta_1, \delta_2} B_{\Omega }(T(z), T(\zeta  ))
   \overline{C_T(\zeta  )}, \\
  {}_{\alpha,\beta}{\bf B}_{\Xi, \gamma_{_T}} = &  J_T  \circ {}_{\delta_1,\delta_2}{\bf B}_{\Omega}\circ\left( J_T\right)^{-1}, 
\end{align*}
where ${} _{\alpha,\beta} B_{\Xi, \gamma_{_T}} $ and ${}_{\delta_1,\delta_2  } B_{\Omega }$ are the Bergman kernels of $ { }_{\alpha,\beta }  A_{\gamma_{_T}} (\Xi) $ and ${ { }_{\delta_1,\delta_2 } A(\Omega) }$, respectively, meanwhile ${}_{\alpha,\beta}{\bf B}_{\Xi, \gamma_{_T}}
$ and $ {}_{\delta_1,\delta_2}{\bf B}_{\Omega}$ are the Bergman projections of ${ }_{\alpha,\beta }  A_{\gamma_{_T}} (\Xi) $ and ${ { }_{\delta_1,\delta_2 } A(\Omega) }$ respectively.
\end{corollary}
\begin{proof} It suffices to use Proposition \ref{ProCon} together with the observation that $C_T$ is a $\mathbb C$-valued function. 
\end{proof}

Finally, we shall introduce another weighted Bergman type spaces in two complex variables induce by the system \eqref{VSistem}.
\begin{definition} 
Given a domain $\Omega\subset\mathbb C^2 \approx \mathbb H$ and $\alpha,\beta\in \mathbb C$. Set $A _{t_{\alpha,\beta}} (\Omega)$ consists of $f\in  {}_{\alpha,\beta}Hol(\Omega, \mathbb C) $ such that 
		\begin{align*}
		\int_{\Omega} |f(z)|^2  t_{\alpha,\beta}(z)  d\mu_z   <\infty  
		\end{align*}where 
		$t_{\alpha,\beta}(z)= e^{   \textrm{Re} ( \bar  \alpha   z_1  + \bar  \beta     z_2) } $ and  $z= z_1+ie^{i  \theta }j z_2$ and define  
	 			\begin{align*}
 		\|f\|_{  A _{t_{\alpha,\beta}} (\Omega)}^2: = &   \int_{\Omega} |f|^2  t_{\alpha,\beta} d\mu, \quad  
	\left\langle f,g\right\rangle_{ A _{t_{\alpha,\beta}} (\Omega)}:=  
	\int_{\Omega} \bar f g  t_{\alpha,\beta}   d\mu   , \\
		{}_{ \alpha, \beta  }\mathcal P [f](z):= &   e^{\textrm{Re} ( \bar \alpha   z_1  + \bar  \beta     z_2) }  f(z) , \quad    \forall  z= z_1+ie^{i  \theta }j z_2\in \Omega ,  
		 \end{align*}
for all $f, g\in A _{t_{\alpha,\beta}} (\Omega)$.   
\end{definition}
The particular case $A_{t_{0,0}}(\Omega)=:A(\Omega)$ was considered in \cite{GM}.
\begin{corollary}   Given $ \alpha, \beta ,\chi,\xi  \in \mathbb C$. Then the operators   
${}_{ \alpha,\beta  }    \mathcal P :  A _{t_{\alpha,\beta}}(\Omega)\to A(\Omega) $ and 
${}_{ \alpha-\chi,  \beta -\xi  }     \mathcal P  :   A _{t_{\alpha,\beta}}(\Omega) \to    A _{t_{\chi,\xi}}(\Omega) $  
are isometric isomorphisms of complex Hilbert spaces and  
\begin{align*} {}_{t_{\chi,\xi}}   B_{\Omega}(z,\zeta )= & 
e^{   \textrm{Re} ( \overline{  (\alpha- \chi) } (z_1+\zeta_1)  + \overline{(  \beta -\xi)}  (  z_2 - \zeta_2)   ) }  
 {}_{t_{\alpha,\beta }}    B_{\Omega}(z,\zeta ) ,\\
 {}_{t_{\chi,\xi}}{ \bf B}_{\Omega} =&  
{}_{\alpha- \chi , \beta-\xi}      \mathcal P  \circ {}_{t_{\alpha,\beta}}{\bf B}_{\Omega}
\circ {}_{      \chi- \alpha,  \xi-\beta     }   \mathcal P,
\end{align*}
    where  
${}_{   \alpha, \beta   }  B_{\Omega}$ and  $ {}_{\alpha, \beta} {\bf B}_{\Omega}$ 
are  the kernel and the projection 
 of  $ \mathcal A _{t_{\alpha,\beta}}(\Omega)$ respectively, and 
 ${}_{   \chi , \xi    }  B_{\Omega}$ and  $ {}_{\chi, \xi} {\bf B}_{\Omega}$ 
are  the kernel and the projection of  $ \mathcal A _{t_{\chi ,\xi }}(\Omega)$.
\end{corollary}
\begin{proof}It follows from  
	Proposition \ref{proo2}.
\end{proof}
\begin{remark}
\begin{enumerate}
\item   
If $$\Omega \subset \{ z= z_1+ie^{i  \theta }j z_2  \in \mathbb H \  \mid \  z_1,z_2\in \mathbb C, \ 
\textrm{Re} ( \bar  \alpha   z_1  + \bar  \beta     z_2)  <0 \}$$
then ${}_{\alpha, \beta} A (\Omega)\subset    A_{t_{\alpha ,\beta}} (\Omega) $ as  function sets.
\item If 
$$\Omega \subset \{ z= z_1+ie^{i  \theta }j z_2  \in \mathbb H \  \mid \  
z_1,z_2\in \mathbb C, \ 
\textrm{Re} ( \bar  \alpha   z_1  + \bar  \beta     z_2)  >0  \}$$
then   $    A_{t_{\alpha ,\beta}} (\Omega)  \subset  
 {}_{ \alpha, \beta } A (\Omega) $  as  function sets. 
\item If $\Omega$ is a bounded domain then  $   A_{t_{\alpha ,\beta}} (\Omega)  =  {}_{ \alpha ,\beta }     A (\Omega) $ as  function sets.
\end{enumerate}
\end{remark}
\section{Acknowledgements}
The authors are thankful to the anonymous referees for the thoughtful and careful evaluation of the manuscript and the many helpful comments and suggestions.
\section{Concluding remarks}
This work is an extension of previous studies developed in \cite{GM, SV3, SV4}. Our contribution here was the design of a methodology for investigation of Bergman spaces in the context of certain perturbed conventional function theories, which are associated with the solutions of first order linear partial differential equation systems embedded into one of the classes of $(\theta, u)-$hyperholomorphic functions. It seems reasonable to expect that this viewpoint will prove to be rather promising and long-range. A discussion on the specific example of $(\alpha, \beta)-$holomorphic functions in two complex variables where such embedding proves to be quite useful and illustrative for the study of Bergman type spaces in domains of $\mathbb C^2$.

There exists another conventional function theory to be considered. This is the case of the Cimmino system, which was introduced in 1941 by G. Cimmino \cite{C}, see also \cite{DL}. The next step to our work would be to develop of Bergman spaces theory induced by the Cimmino system, whose solutions form a proper subset of the $\displaystyle\frac{\pi}{2}-$hyperholomorphic functions class.

\section*{Declarations}
\subsection*{Funding} Instituto Polit\'ecnico Nacional (grant number SIP20211188) and CONACYT.}
\subsection*{Conflict of interest} The authors declare that they have no conflict of interest regarding the publication of this paper.
\subsection*{Author contributions} Both authors contributed equally to the manuscript and typed, read, and approved the final form of the manuscript, which is the result of an intensive collaboration.
\subsection*{Availability of data and material} Not applicable
\subsection*{Code availability} Not applicable

\end{document}